\documentclass[11pt, letter]{article}
\usepackage[letterpaper,left=1.in, right=1.in, top = 1.truein, bottom = 1.truein]{geometry}

\usepackage{eulervm}
\usepackage{tgpagella}
\usepackage[T1]{fontenc}
\usepackage{amsmath}
\usepackage{amssymb}
\usepackage{amsthm}
\usepackage{cite}
\usepackage{color}
\usepackage{graphicx}
\usepackage{caption}
\usepackage{color}
\usepackage{bm}
\usepackage{mathtools}

\DeclareMathAlphabet{\mathpzc}{OT1}{pzc}{m}{it}

\DeclareMathOperator{\rank}{rk}

\usepackage{algorithm}
\usepackage{algpseudocode}
\makeatletter \def\BState{\State\hskip-\ALG@thistlm} \makeatother

\usepackage{etoolbox} \makeatletter
\patchcmd{\@maketitle}{\begin{center}}{\begin{flushleft}}{}{}
\patchcmd{\@maketitle}{\begin{tabular}[t]{c}}{\begin{tabular}[t]{@{}l}}{}{}
\patchcmd{\@maketitle}{\end{center}}{\end{flushleft}}{}{}

\newcommand\be{\begin{equation}}
\newcommand\ee{\end{equation}}

\newtheorem{thm}{Theorem}[section]
\newtheorem{lem}[thm]{Lemma}
\newtheorem{prp}[thm]{Proposition}

\newtheorem{dfn}[thm]{Definition}
\newtheorem{rem}[thm]{Remark}
\newtheorem{obs}[thm]{Observation}

\newcommand{\Real}{\mathbb{R}}

\newcommand{\Z}{\mathbb{Z}}

\newcommand{\V}{\mathcal{V}}

\def\Real{\mathbb{R}}

\newcommand{\ph}{\mathbf{PH}}
\newcommand{\homol}{\mathbf{H}}

\newcommand\path{\bm{\pi}}

\DeclarePairedDelimiterX\set[1]\lbrace\rbrace{\def\suchthat{\; \delimsize\vert\;}#1}

\newcommand{\restr}[1]{|_{#1}}
\newcommand{\filt}[1]{\ensuremath{M_{\gamma(#1)}}}
\newcommand{\filtf}[1]{\ensuremath{F_{#1}}}
\newcommand{\filtg}[1]{\ensuremath{G_{#1}}}

\newcommand{\pd}{\ensuremath{PH}}

\newcommand{\grad}{\nabla}

\newcommand{\cell}[1]{\ensuremath{e^{#1}}}
\newcommand{\homotopeq}{\ensuremath{\simeq}}
\newcommand{\lifecont}[2]{\ensuremath{PG_{#1}^{#2}}}
\newcommand{\lcpiece}{\ensuremath{l}}
\newcommand{\hittime}[2]{\ensuremath{t_{#1}^{#2}}}

\newcommand{\vecl}{\ensuremath{\prec}}

\newcommand{\ra}{\ensuremath{\rightarrow}}
\newcommand{\R}{\ensuremath{\mathbb{R}}}
\newcommand{\fnbd}{\ensuremath{l}}
\newcommand{\Fnbd}{\ensuremath{L}}

\newcommand{\sing}{\bm{\Sigma}}
\newcommand{\phd}{\ph}

\usepackage{algorithm}
\usepackage[shortlabels]{enumitem}
\usepackage{algpseudocode}
\usepackage{hyperref}
\usepackage{tikz}
\usetikzlibrary{backgrounds}
\usetikzlibrary{intersections}
\usetikzlibrary{positioning}

\usepackage[margin=1cm]{caption}

\makeatletter \def\BState{\State\hskip-\ALG@thistlm} \makeatother

\usepackage{etoolbox} \makeatletter
\patchcmd{\@maketitle}{\begin{center}}{\begin{flushleft}}{}{}
\patchcmd{\@maketitle}{\begin{tabular}[t]{c}}{\begin{tabular}[t]{@{}l}}{}{}
\patchcmd{\@maketitle}{\end{center}}{\end{flushleft}}{}{}

\usepackage{graphicx}
\usepackage{subcaption}
\usepackage{layouts}

\begin{document}


\title{Biparametric persistence for smooth filtrations \footnotetext{The authors were supported by the ARO through the MURI
'Science of Embodied Innovation, Learning and Control'.} } 


\author{Mishal Assif P K$^1$, Yuliy Baryshnikov$^{1,2,3}$ }
\date{%
    $^1$University of Illinois, Department of ECE\\%
    $^2$University of Illinois, Department of Mathematics\\
    $^3$Kyushu university, IMI\\[2ex]%
    \today
}

\maketitle
\abstract{The goal of this note is to define biparametric persistence diagrams for smooth generic mappings \(h=(f,g):M\to V\cong \Real^2\) for smooth compact manifold \(M\). Existing approaches to multivariate persistence are mostly centered on the workaround of absence of reasonable algebraic theories for quiver representations for lattices of rank 2 or higher, or similar artificial obstacles. We approach the problem from the Whitney theory perspective, similar to how single parameter persistence can be viewed through the lens of Morse theory.}
\section{Introduction}
\subsection{Persistence}
We recall that the ``classical'', one-parametric persistence deals with one parameter filtrations, 
an $\R \text{ or } \Z$-indexed collection of subsets $\{X_s\}_{s\in I}$ of a topological space
$X$ such that $\cup_{s \in I}X_s = X$ and
\[
    X_s \subset X_t, \quad \text{if } s \leq t.
\]
If $s \leq t$, the natural inclusions
\[
    i_{s, t}:X_s \rightarrow X_t
\]
induce linear maps
\[
    \left(i_{s, t}\right)_*: \homol_k(X_s) \rightarrow \homol_k(X_t)
    \]
    (all homology groups here are over a field), so that 
satisfying the natural consistency (``functoriality''): for $r \leq s \leq t$
\[
    i_{s, t} \circ i_{r, s} = i_{r, t}.
    \]
    
This collection of vector spaces $\{\homol_k(X_s)\}_{s \in I}$
connected by linear maps $\{ \left(i_{s, t}\right)_*\}_{s\leq t}$, is called the 
persistence module associated with the filtration.

One-parametric persistence theory asserts that upto isomorphism, this persistence
module is completely characterized by a set of points in the plane called the persistence
diagram \cite{zomo2005}. Essentially, the whole collection of homology mappings split into chains of isomorphisms, whose terminal points, - birth and death points, - carry a lot of information about the filtration.

This theory, appearing in different guises essentially since Morse work, provided an essential toolbox for many researchers dealing with filtrations, from geometric analysis to data analysis. The diagrams of birth-death pairs, known as the {\em persistence diagrams} 
are known to be stable with respect to perturbations in a certain sense \cite{cohen2007}, and can be computed algorithmically \cite{harer2010}.

Biparametric persistence deals with the filtrations indexed by points in the plane, $\R^2 \text{ or } \Z^2$ endowed with the usual product order relation.

\subsubsection{Slices and Bipersistence}
In the smooth category, the following definition is natural.

Let $h = (f, g):M \rightarrow \V\cong \Real^2$ be a map from the topological space
$M$ to the plane.
We define the {\em $c=(a,b)$-slice} of a manifold as the set
\[
M_c=M_{a,b}=\{f\leq a, g\leq b\}.
\]
We also denote the sublevel sets of $f$ and $g$ as
\[
    F_a = \{f \leq a \}, \quad G_b = \{g \leq b\}.
\]
We say $c_1 = (a_1, b_1) \prec c_2 = (a_2, b_2)$ when
\[
    a_1 \leq a_2, \quad b_1 \leq b_2.
\]
The sublevel sets $\{F_a\}_{a \in \Real}$ form a one-parameter filtration of $M$,
while the slices $\{M_c\}_{c \in \Real^2}$ give rise to a biparametric filtration
on $M$. The persistence module associated with such bi-filtrations are the central object of our study.

\subsection{Context: Biparametric Persistence}
The goal of this note is to define the biparametric persistence diagrams for smooth generic mappings \(h=(f,g):M\to V\cong \Real^2\) for smooth compact manifold \(M\).

Existing approaches to multivariate persistence are mostly centered on the workaround of absence of reasonable algebraic theories for quiver representations for lattices of rank 2 or higher, or similar artificial obstacles, and thus are focused on the discrete filtrations \cite{lesnick, lesnick_computing_2019,miller_data_nodate,schenck}.

An alternative thread of research dealing with the biparametric persistence relies on restricting the filtration to straight lines with a positive slope in the plane of parameters, and studying the resulting (two-parametric) family of one-parametric persistence diagram, see \cite{cerri2015necessary,cerri_multidimensional_2009,Cer19}.

There are other treatises of biparametric persistence, exploring the space adjacent to the major threads we mentioned, see e.g. \cite{vipond_multiparameter_nodate,scaramuccia2020computing}.

    Our view on the problem was driven by the parallels between the Morse theory and persistence theory: the latter is a subtle globalization of the former. The corresponding singularity theory for the biparametric persistence theory would deal with the critical points of smooth mappings into $\Real^2$, a classical topic, going back to H. Whitney \cite{Whi55}. What is the persistence theory for such mappings? This was the primary motivating question guiding this research.

    Another guiding idea was to dramatically expand the (natural) approach of describing the persistence structures on the increasing one-parametric filtrations induced from the biparametric ones. To this end we replace the family of straight lines with positive slopes deployed by Frosini and his co-authors to the functional family of arbitrary increasing curves. While the resulting family of curve is infinite dimensional, this approach allows one to keep track of all patterns of the biparametric filtrations. Moreover, the space of increasing curves can be effectively discretized, leading to a finite (for compact source manifold) cubical CAT($0$) complex \cite{sageev_ends_1995} with the one-dimensional persistence data attached to its vertices, giving a complete description of the biparametric diagram.

\subsection{Singularities of mappings into the plane}
Recall that Morse theory asserts that a generic real valued function on a compact
manifold has its critical points nondegenerate, isolated, and having different
critical values. In addition, the local form of such a function in a neighborhood of its
critical point is completely determined by a single integer, called the index, associated
with each critical point. This gives significant information about the evolution of
topology of the slice sets of such a function.

Whitney theory \cite{Whi55} parallels Morse theory for the mappings into the plane, and we will
use it to understand the evolution of the homology of the slice sets.

\begin{dfn} Singular set $\sing(h)\subset M$ of $h$ is the set of points $x\in M$  where the rank of the Jacobian $D_xh: T_xM\to\Real^2$ is less than maximal (i.e. less than $2$).

The image $h(\sing(h))$ is called the {\em visible contour} of $h$.
\end{dfn}

Whitney established that for a generic map (i.e. a map belonging to an open dense subset of smooth maps), the rank of the Jacobian does not drop to zero, and the singular set is a smooth curve in $M$. Moreover, at the singular points, an appropriate choice of coordinates in $M$ and $V$ brings $h$ to one of the two canonical forms:

\begin{enumerate}
    \item $f(x_1, ..., x_m) = x_1, \quad g(x_1, ..., x_m) = q(x_2, ..., x_m),$  
    \item $f(x_1, ..., x_m) = x_1, \quad g(x_1, ..., x_m) = x_2^3 + x_1x_2 + q(x_3, ..., x_m),$  
\end{enumerate}
where $q$ is a non-degenerate quadratic function. The points of the first kind are called the {\em fold points},
and $h$ restricted to $\sing(h)$ is an immersion at such points.

The points of the second type are isolated (and therefore, there is only finitely many of them, as $M$ is compact).
Such points are called {\em cusp points} which are characterized by the fact that the
Jacobian of $h$ restricted to $\sing(h)$ degenerates.

Therefore $h$ is an immersion of $\sing(h)$ at the fold points (so its image is a smooth curve, perhaps with self-intersections), and the cusp points map to cusps on the visible
contour.

\begin{dfn}
    The $h$-images of critical points of $f$ are called {\em vertical}, of $g$, {\em horizontal} points.
\end{dfn}

Applying the standard arguments (see, e.g. \cite{stable}) leads to the following

\begin{prp} For generic $h$ the set of critical points $\sing(h)$ is a smooth curve in $M$,
mapping to a curve with cusps and simple self-intersections in $V$,
and the component functions $f,g$ are Morse, with the critical points
distinct from the cusp points, and map to different points by $h$.

Further, the tangents to the $h$-image of the fold points is vertical or horizontal exactly at the vertical or horizontal points, the curvatures of the visible contour at the critical points of $f$ and $g$ are non-vanishing, and the the points of the self-intersection of the fold curve are neither horizontal nor vertical.
\end{prp}

As the rank of the Jacobian is $1$ at the fold points, there is a unique (up to a multiple) linear combination of the differentials of $f$ and $g$ vanishing there. If the coefficients are of the same sign, we will say that the point has {\em negative slope}, otherwise, {\em positive slope}.

The nonvanishing of the curvature of the contour at the fold points (in particular, at the vertical and horizontal points) established for the generic maps implies that vertical and horizontal points split the fold curve into a finite umber of the alternating segments of positive and negative slope points.

\section{Pareto Grid}
\label{sec:sec3}

\subsection{Pareto Points and Extension Rays}
Previous applications of the Whitney theory dealt with the simultaneous optimization of several functions (important in economic theory), see e.g. \cite{smale75}, motivating the following nomenclature:
\begin{dfn}
The (closure) of contour points of negative slope is called the set of {\em Pareto points}: they form a collection of curves with cusps and self-intersections with boundaries at horizontal and vertical points.
\end{dfn}

Augment the set of Pareto points with the union of vertical and horizontal rays, attached to the (correspondingly) vertical and horizontal points, and such that the respective coordinate increases to $+\infty$ along the ray.

\begin{dfn}
We will be referring to such rays as the {\em extension rays}, and the union of Pareto segments and the extension rays in the $(f,g)$ plane as the {\em Pareto grid}.
\end{dfn}

(We borrowed the term {\em Pareto grid} from \cite{Cer19}, who introduced this notion independently in 2019.)

\begin{figure}
\centering
\begin{subfigure}{0.3\linewidth}
  \centering
  \includegraphics[width=.8\textwidth]{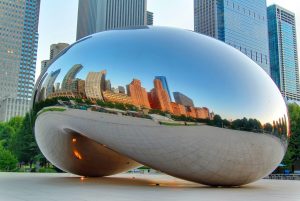}
  \caption{Chicago Millenium Park}
  \label{fig:sub1}
\end{subfigure}
\begin{subfigure}{.28\linewidth}
    \centering
  \includegraphics[width=.6\textwidth]{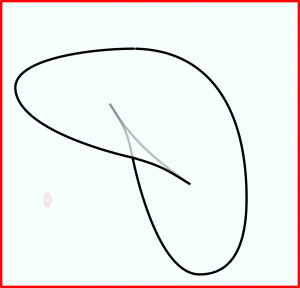}
  \caption{Visible contour}
  \label{fig:sub2}
\end{subfigure}
\begin{subfigure}{0.28\linewidth}
  \centering
\includegraphics[width=0.6\textwidth]{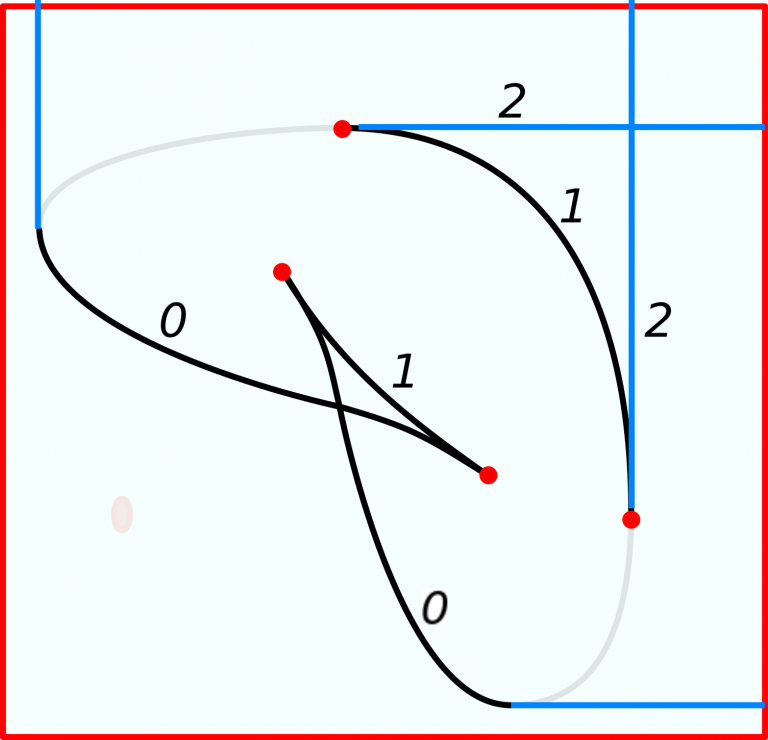}
\caption{Pareto grid} 
\label{fig:Pareto grid}
\end{subfigure}
\caption{\small Figure (a) shows a bean sculpture, the surface of which is a 2 dimensional sphere.
If $h$ is the projection of the surface onto a plane behind the bean the corresponding
visible contour is shown in Figure (b). The Pareto grid, along with cusps, pseudocusps and
indices of various segments of the grid are shown in Figure (c).}
\label{fig:contour}
\end{figure}

\begin{dfn}
    A boundary point of a Pareto segment is called a pseudocusp if the extension ray
    does not attach smoothly with the Pareto segment at the point.
\end{dfn}

An illustration of all these definitions is given in Figure \ref{fig:contour}.
The Pareto grid shown on the right display has two cusps and two pseudocusps
indicated as points in red. The blue lines are the extension rays and the black curves
are the Pareto segments. There is one vertical and one horizontal point on the Pareto
grid which are not pseudocusps as the extension rays at these points fit smoothly
with the Pareto segment.

\subsection{Topology of Slices}

In biparametric persistence, the Pareto grid plays a role analogous to set of critical values of a Morse function in one-parametric persistence theory. Outside of it, the topology of the slices does not change, and across it, it changes in a controllable way.

\begin{thm}
    \label{th:homotopic slice}
  The slices $M_c$ are homeomorphic as $c$ varies within an (open) connected component of the complement to the Pareto grid.
\end{thm}

Crossing the Pareto grid leads to easy to interpret changes in the topology of the slices.

\begin{thm}
    \label{th:attach slice}
  For each pair of ordered points in $V$, such that a path connecting them intersects
  the Pareto grid transversally at a single point which is not a cusp or a pseudo-cusp, the higher slice is homotopy equivalent
  to the lower one with a single cell of dimension $k$ attached by its boundary.
\end{thm}

\begin{proof}[Proof of Theorem \ref{th:homotopic slice}]
  We need to prove that $M_c$ is homeomorphic to $M_{c'}$ for $c'$ close enough to $c$. It is clearly enough to prove that for $c'$ differing from $c$ just in one coordinate: say, $c=(a,b); c'=(a',b)$. By definition, the hypersurfaces $\{f=a\}$ and $\{g=b\}$ in $M$ are smooth, and intersect transversally near the slice $M_c$. Now, the result follows immediately from collaring theorems for the manifolds with boundaries and submanifolds intersecting them transversally (see, e.g. \cite{hirsch2012}).
\end{proof}

\begin{proof}[Proof of Theorem \ref{th:attach slice}]
  Let $c$ be a point of the Pareto grid which is not a (pseudo)-cusp. As in the proof of Theorem \ref{th:attach slice}, it is enough to consider a pair of points close to $c=(a,b)$ that differ from $c$ only in one coordinate, say $a'<a<a''$. In this case we consider the  change in topology of the sublevel set of a Morse function $f$ on the manifold with boundary given by $G_b=\{g\leq b\}$. This is, of course, a well-established subject, see e.g. \cite{arnold1976,goresky}.

  We need to distinguish two cases:

  \begin{itemize}
  \item If $a$ is the critical value of $f$ corresponding to a critical point $x$ in the interior of $G_b$, the point $c$ is located on the extension ray. In this case, the claim follows immediately, as by the assumptions, $f$ is a Morse function, and all local changes of the topology near $x$ amount to attaching a cell of the dimension equal to the index of $f$ at $x$. (Outside of some vicinity of $x$, the collaring theorem applies again.)
  \item If $a$ is the critical value of the restriction of $f$ to the boundary $\{g=b\}$ of $G_b$, then, again by genericity, the corresponding critical point $x$ is Morse, and we are locally dealing with, in an appropriate chart with the with the change of topology of the sets
    \[
\bm{M}_\alpha:=x_1\geq 0; x_1\leq \phi(x_2,\ldots,x_m)+\alpha
\]
as $\alpha$ varies across $0$, and $\phi$ is a Morse function (of the same index $k$ as the restriction of $f$ to $G_b$ at $x$). It is immediate, that locally, the change of topology amounts to addition of a cell of index $k$.
  \end{itemize}
  This proves the result.
\end{proof}

\subsection{Indices on the Grid}
The results of the previous section allow us to attach indices to the component curves of the Pareto grid.

\begin{dfn}We will refer to the dimension $k$ of the cell attached when crossing a point on
  the Pareto grid as the {\em index} of the point.
\end{dfn}

It is equal, to remind, to the index of $f$ or $g$ at the critical point corresponding to an extension ray, and to the index of the restriction of one of the functions to the level set of the other, for the Pareto points of the grid.

It is known that attaching a $k$-cell to a space has the effect of either increasing
the dimension of $\homol_k$ by one or decreasing that of $\homol_{k-1}$ by one.
This can be phrased algebraically as 
\begin{enumerate}
    \item $(i_{\gamma(a), \gamma(b)})_*: \homol_k\left(M_{\gamma(a)}\right) \rightarrow \homol_k\left(M_{\gamma(b)}\right)$ 
        is injective with cokernel of dimension 1,
    \item $(i_{\gamma(a), \gamma(b)})_*: \homol_{k-1}\left(M_{\gamma(a)}\right) \rightarrow \homol_{k-1}\left(M_{\gamma(b)}\right)$ 
        is surjective with kernel of dimension 1.
\end{enumerate}

This implies, inter alia, that the index is constant along the segments of the Pareto grid outside of cusps or pseudo-cusps:
\begin{prp}
    \label{p:ind def}
  The dimensions of the attached cells are constant along the smooth immersed components of the
  Pareto grid, the complements to the (pseudo)cusps.
\end{prp}

\begin{proof}
  Consider first the smooth components of the Pareto grid away from the cusps or pseudocusps. We know (by Theorem  \ref{th:homotopic slice}
  that the topology is constant on either side of that smooth curve. Hence, the changes of the topology should be the same wherever on crosses it.

  The situation near the intersecting components of the Pareto grid is also clear: the changes of the topology caused by the crossing of either of the components are localized near the corresponding critical points, which are distinct, and therefore are independent of each other. 
\end{proof}

Another implication of the Proposition \ref{p:ind def} allows us to characterize the indices of the branches of the Pareto grid connecting at a (pseudo)cusp:
the change in index of a smooth segment of the Pareto grid at a (pseudo)cusp is 
also characterized in the following result.

\begin{prp}
    \label{p:ind change}
At the (pseudo)cusps, the points of the two branches are naturally ordered: any point of one branch near the (pseudo) cusp is greater than (some of the) points of the other (we will call the former branch the {\em upper}, the latter, the {\em lower} one). The index of the upper branch exceeds the index of the lower branch by one.
\end{prp}

\begin{figure}
\centering
\begin{subfigure}{0.4\linewidth}
  \centering
  \includegraphics[width=.5\linewidth]{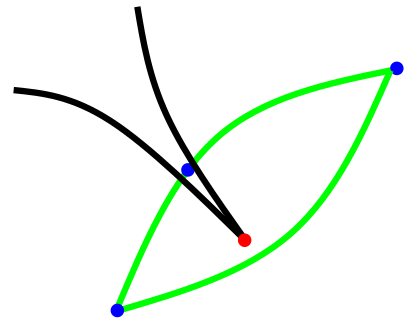}
  \caption{neighborhood of a cusp}
  \label{fig:cusp}
\end{subfigure}
\begin{subfigure}{.4\linewidth}
  \centering
  \includegraphics[width=.5\linewidth]{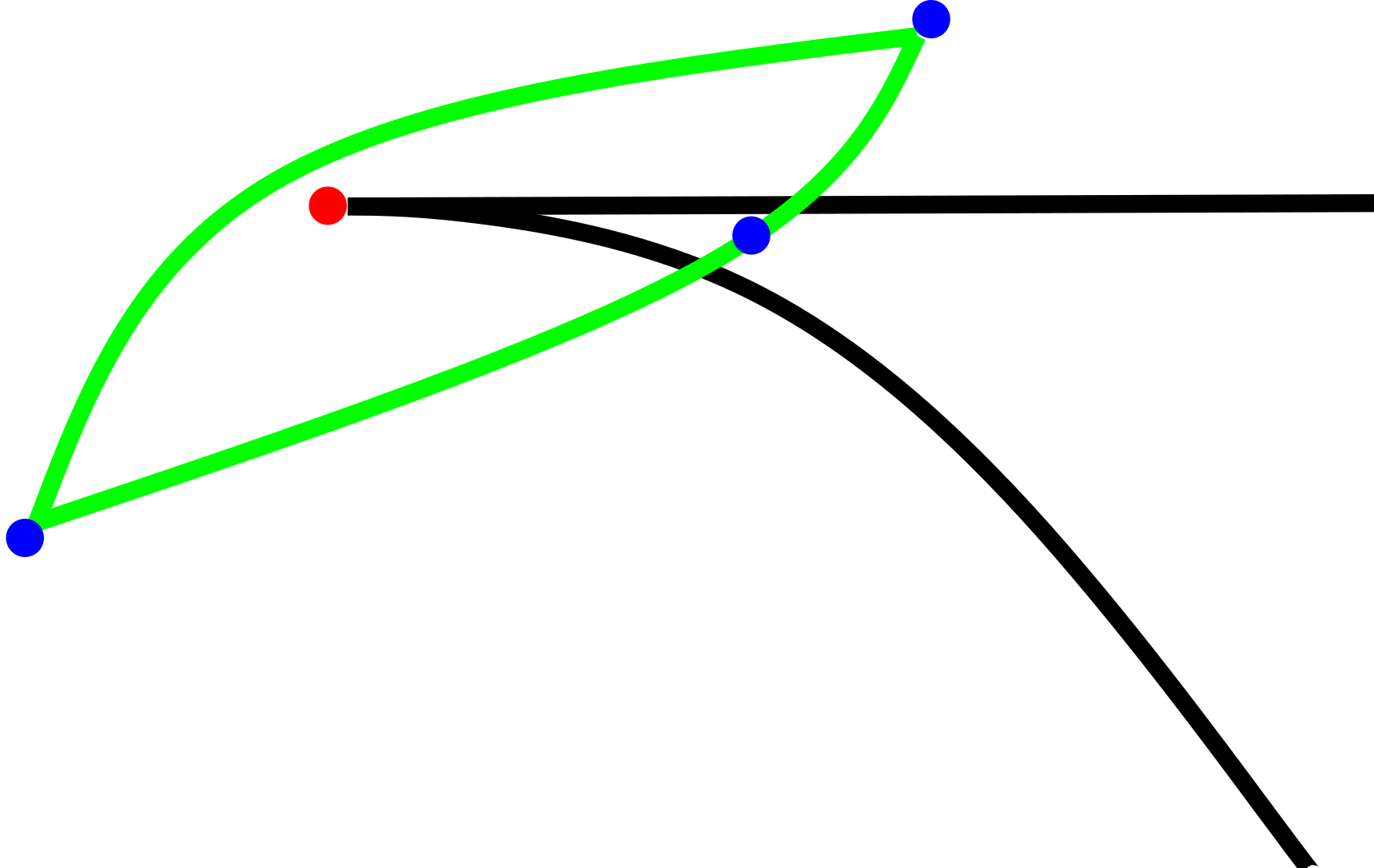}
  \caption{neighborhood of a pseudocusp}
  \label{fig:pseudocusp}
\end{subfigure}
\caption{Curves in a neighborhood of a (pseudo)cusp.}
\label{fig:pf1}
\end{figure}

\begin{proof}
  Consider a pair of points outside of the Pareto grid, one slightly above, one slightly below the cusp. As they belong to the same connected component of the complement to the Pareto grid, the corresponding slices are homeomorphic. One the other hand, one can find an increasing curve connecting these two points, intersecting the Pareto grid at two branches joined at the cusp. At each of the branches a disk of some dimension is attached to the slice, resulting in trivial change of the topology. By inspection, this is possible only if at the upper branch, the class generated by the lower one was annihilated, implying the result.
\end{proof}

\section{Increasing curves and persistence}
\label{sec:sec4}
We will be referring to the pseudocusps and cusps as {\em obstacles}.
Consider an {\em increasing curve} \(\gamma:\Real\to\Real^2\) in the \(f,g\)-plane
(this means that both functions strictly increase along the curve). To fix the
gauge, we will assume that the curve is parameterized by \(f+g\)
(we will refer to this parameter as the {\em natural height}).

An increasing curve \(\gamma\) defines a usual \(I\)-indexed filtration of \(M\), 
and correspondingly, persistent homologies and persistent diagrams \(\phd_k(\gamma), k=0,\ldots,m\),
which we interpret as a collections of distinguishable points in planes \(\{b<d\}\).
The results of the previous section describe the nature of these persistence diagrams.
When the curve hits an index $k$ segment of the Pareto grid, a $k$-cell is attached
to the slice set and so either the dimension of the $k$-th homology of the slice
set increases by one or that of the $k-1$-th homology decreases by one.
Therefore certain pairs of intersections of the curve $\gamma$ with an index $k$ segment
and an index $k+1$ segment leads to a birth-death pair in the $k$-th homology of the slice
set, and such pairs collected together form the persistence diagram $\phd_k(\gamma)$.

\begin{rem}
The approach to multiparametric persistence through restriction to
increasing \textit{straight} lines has been considered by \cite{Cer19, lesnick}.
\end{rem}

We will see that we can completely describe the change in persistence diagram
as the curve varies in the set of increasing curves with fixed endpoints
avoiding the obstacle points. It is natural to understand this space
of increasing curves first.

\begin{prp} In a generic family of increasing paths (and for a generic $h$), the sets of paths passing through an obstacle is a smooth hypersurface; the hypersurfaces corresponding to the different obstacles intersect transversally. 
\end{prp} 
The hypersurfaces are naturally cooriented (by the increase of number of intersections with the Pareto grid).

We say that increasing curves $\gamma$ and $\gamma^{'}$ connecting the same endpoints
belong to the same path-connected collection if there exists a homotopy of increasing curves
fixing the endpoints connecting the two curves.
The following sequence of results describe how the persistent homology changes as
the curve $\gamma$ moves in the space of increasing paths.

\begin{thm}For the path-connected collection of increasing curves in the plane avoiding obstacles, there is a section of the space of persistent diagrams: that is for any two such curves \(\gamma,\gamma'\), there is an identification
\[
I(\gamma',\gamma):\phd_*(\gamma)\to\phd_*(\gamma'),
\]
of the bars in the persistent diagrams corresponding to each of the curves, and these identifications are consistent:
\[
I(\gamma'',\gamma')I(\gamma',\gamma)=I(\gamma'',\gamma).
\]
\end{thm}

\begin{proof}

    We denote by $\lifecont{}{} \supset \lifecont{k}{} \supset \lifecont{k}{\gamma}$
    the set of smooth segments of the complement of obstacles of the Pareto grid, 
    the set of smooth segments of index $k$ and the set of smooth
    segments of index $k$ and hitting the increasing curve
    $\gamma$ respectively. If $\lcpiece \in \lifecont{k}{\gamma}$, we will denote by $\hittime{\lcpiece}{\gamma}$
    the unique time at which $\gamma$ hits $\lcpiece$. 
    
    If $\gamma$ and $\gamma'$ satisfy the conditions of the theorem,
    $\lifecont{k}{\gamma} = \lifecont{k}{\gamma'}$ for all $k$. Only
    when a curve crosses either a pseudocusp or a cusp can it hit a new segment of the
    Pareto grid.

    As far as $\phd_k$ is concerned, only the index $k$ and $k+1$ segments matter.
    So we prove this result in three stages, assuming the curves $\gamma$ and $\gamma'$
    can be connected to each other by a homotopy

    \begin{enumerate}
        \item avoiding the obstacles and intersection points of two segments
            from $\lifecont{k}{}\cup\lifecont{k+1}{}$,
        \item avoiding the obstacles and intersection points of two segments
            of the same index,
        \item avoiding the obstacles only.
    \end{enumerate}

    \textbf{Case 1:} In this case, it is not only true that $\lcpiece \in \lifecont{k}{\gamma} = \lifecont{k}{\gamma'}$
    but the order in which the two curves hit each smooth segment is also the same.
     The recipe to construct the identification $I_k(\gamma', \gamma)$
    between the persistence diagrams of $\gamma$ and $\gamma'$ is to simply
    map $(\hittime{\lcpiece_1}{\gamma}, \hittime{\lcpiece_2}{\gamma}) \in \pd_k(\gamma)$
    to $(\hittime{\lcpiece_1}{\gamma'}, \hittime{\lcpiece_2}{\gamma'}) \in \pd_k(\gamma')$.
    The consistency of this identification follows obviously. We just need to show that
    $(\hittime{\lcpiece_1}{\gamma}, \hittime{\lcpiece_2}{\gamma}) \in \pd_k(\gamma)
    \implies (\hittime{\lcpiece_1}{\gamma'}, \hittime{\lcpiece_2}{\gamma'}) \in \pd_k(\gamma')$.
    
    We show this assuming $\gamma'$ lies in an $\epsilon$ tube around $\gamma$ for small enough 
    $\epsilon$. The general result will follow since $\gamma$ and $\gamma'$ can be connected
    by finitely many such tubes. If the $\epsilon$ is small enough, we can choose
    points $\gamma(t^{\pm}_{l_i})$ just after and before $\gamma(t_{l_i})$ for $i=1,2$.
    We can do the same for $\gamma'$, and then choose $c^{-}_i$ close to $\gamma(t_{l_i})$
    such that $c^{-}_i \prec \gamma(t^{-}_{l_i})$ and $c^{-}_i \prec \gamma'(t^{-}_{l_i})$,
    and a $c^{+}_i$ also close to $\gamma(t_{l_i})$ such that 
    $c^{+}_i \succ \gamma(t^{+}_{l_i})$ and $c^{+}_i \succ \gamma'(t^{+}_{l_i})$.
    These choices are illustrated in Figure \ref{fig:id pd fig 1}. If $\epsilon$ is
    close enough, the inclusions $i_{c^{-}_i, \gamma^{-}_{t_i}}, i_{c^{-}_i, \gamma'^{-}_{t_i}},
    i_{\gamma^{+}_{t_i},c^{+}_i}, i_{\gamma'^{+}_{t_i},c^{+}_i}$ are all homotopy
    equivalences.

    The existence of a bar between $t_{l_i}$ and $t_{l_j}$ along $\gamma$ is equivalent to
    the existence a homology class $\alpha$ in $\homol_{k}(\filt{t^{+}_{l_i}})$ that
    is:
    \begin{enumerate}[(i)]
        \item not in the image of $i_{\gamma(t^{-}_{l_i}), \gamma(t^{+}_{l_i})}$,
        \item whose image under $i_{\gamma(t^{+}_{l_i}), \gamma(t^{-}_{l_j})}$ is not in
            the image of $i_{\gamma(t^{-}_{l_i}), \gamma(t^{-}_{l_j})}$,
        \item and whose image under $i_{\gamma(t^{+}_{l_i}), \gamma(t^{+}_{l_j})}$ is in the
            image of $i_{\gamma(t^{-}_{l_i}), \gamma(t^{+}_{l_j})}$.
    \end{enumerate}
    We can push this class $\alpha$ via the isomorphism $i_{\gamma(t^{+}_i), c^{+}_i}$ to obtain
    a homology class that is:
    \begin{enumerate}[(i)]
        \item not in the image of $i_{c^{-}_i, c^{+}_i}$,
        \item whose image under $i_{c^{+}_i, c^{-}_j}$ is not in the image of $i_{c^{-}_i, c^{-}_j}$,
        \item and whose image under $i_{c^{+}_i, c^{+}_j}$ is in the image of $i_{c^{-}_i, c^{+}_j}$,
    \end{enumerate}
    owing to the fact that each $c^{\pm}_{i/j}$ point and the corresponding $t^{\pm}_{i/j}$
    point has been chosen such that inclusion induces an isomorphism in the $k$-th homology
    between them. This same procedure can be used to push this class onto the corresponding
    $t^{'\pm}_{i/j}$ points to obtain a homology class that is born at $t^{'}_i$ and
    dies at $t^{'}_j$ along $\gamma'$. This proves the result for case 1.

    \begin{figure}[ht]
        \centering
        \begin{tikzpicture}[every node/.style={black,right}]

            \draw[ultra thick, name path=lc] (0,6) .. controls (3,3) and (6,1) .. (10,0);
            \draw[ultra thick, green, name path=gamma] (0,1) .. controls (3,3) and (5,4) .. (10,5.5);
            \draw[ultra thick, green, name path=gamma2] (0,0.5) .. controls (5,3.5) .. (10,6);

            \path [name intersections={of=lc and gamma,by=ti}];
            \node (tin) [circle,fill=purple,inner sep=1.5pt,label={}] at (ti) {};
            \path [name intersections={of=lc and gamma2,by=t2i}];
            \node (t2in) [circle,fill=purple,inner sep=1.5pt,label={}] at (t2i) {};
            \node (tinlabel) [label={}, above=1.5cm of tin] {$\gamma(t_{l_i})$};
            \draw[->, thick, blue] (tin) -- (tinlabel);
            \node (t2inlabel) [label={}, below=1.5cm of t2in] {$\gamma'(t^{'}_{l_i})$};
            \draw[->, thick, blue] (t2in) -- (t2inlabel);
            
            \path[name path=Vmin] (2.75,0) -- (2.75,5.5);
            \path [name intersections={of=gamma and Vmin,by=t-i}];
            \node (tminin) [circle,fill=purple,inner sep=1.5pt,label={}] at (t-i) {};
            \path [name intersections={of=gamma2 and Vmin,by=t2-i}];
            \node (t2minin) [circle,fill=purple,inner sep=1.5pt,label={}] at (t2-i) {};
            \node (tminlabel) [label={}, above=1.5cm of tminin] {$\gamma(t^{-}_{l_i})$};
            \draw[->, thick, blue] (tminin) -- (tminlabel);
            \node (t2minlabel) [label={}, below=1.5cm of t2minin] {$\gamma'(t^{'-}_{l_i})$};
            \draw[->, thick, blue] (t2minin) -- (t2minlabel);

            \path[name path=Vplus] (4.6,0) -- (4.6,5.5);
            \path [name intersections={of=gamma and Vplus,by=t+i}];
            \node (tplusin) [circle,fill=purple,inner sep=1.5pt,label={}] at (t+i) {};
            \path [name intersections={of=gamma2 and Vplus,by=t2+i}];
            \node (t2plusin) [circle,fill=purple,inner sep=1.5pt,label={}] at (t2+i) {};
            \node (tpluslabel) [label={}, above=1.5cm of tplusin] {$\gamma(t^{+}_{l_i})$};
            \draw[->, thick, blue] (tplusin) -- (tpluslabel);
            \node (t2pluslabel) [label={}, below=1.5cm of t2plusin] {$\gamma'(t^{'+}_{l_i})$};
            \draw[->, thick, blue] (t2plusin) -- (t2pluslabel);

            \draw[fill=cyan, opacity=0.2] (tin) circle [radius=3];
            \draw[fill=yellow, opacity=0.2] (tin) circle [radius=1.5];
            \node (radius) [label={90:radius: $ \epsilon$},left = 2.95cm of tin] {};
            \draw[->, thick, blue, dashed] (tin) -- (radius);

            \node (cplus) [circle,fill=purple,inner sep=1.5pt,label={},above right = 1cm of t+i] {};
            \node (cpluslabel) [label={}, above right=1cm of cplus] {$c^{+}_{i}$};
            \draw[->, thick, blue] (cplus) -- (cpluslabel);
            \node (cmin) [circle,fill=purple,inner sep=1.5pt,label={},below left = 1cm of t2-i] {};
            \node (cminlabel) [label={}, below left=1cm of cmin] {$c^{-}_{i}$};
            \draw[->, thick, blue] (cmin) -- (cminlabel);

        \end{tikzpicture}
        \caption{The green increasing lines indicate the curves $\gamma$ and $\gamma'$
            and the black line with negative slope is the piece of the life contour $\lcpiece$.
            The blue disk around $\gamma(t_{l_i})$ is a possible choice for the $\epsilon$ tube.
            Also indicated are possible choices for
            $t^{\pm}_{l_i},t^{'\pm}_{l_i},c^{\pm}_{i}$.
        }
        \label{fig:id pd fig 1}
    \end{figure}
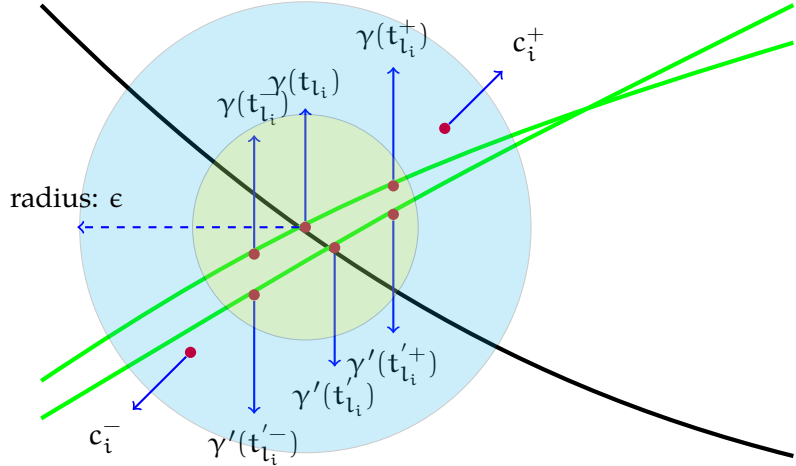
    
    \textbf{Case 2:} In this case, there can be intersections between an index $k$
    and an index $k+1$ segment of the grid. This situation is illustrated in Figure 
    \ref{fig:id pd fig 2}. We have two segments $l_l, l_r$ of the Pareto grid intersecting
    at a point. We can also just show the result for curves $\gamma$ and $\gamma'$ 
    that are equal outside a neighborhood of the intersection point. Suppose
    $l_l$ has index $k$ and $l_r$ has index $k+1$. The same mapping in case 1,
    sending $(\hittime{\lcpiece_1}{\gamma}, \hittime{\lcpiece_2}{\gamma}) \in \pd_k(\gamma)$
    to $(\hittime{\lcpiece_1}{\gamma'}, \hittime{\lcpiece_2}{\gamma'}) \in \pd_k(\gamma')$
    will work in this case as well. All that needs to be shown is that a bar born
    along $l^-_l$ will not die at $l^-_r$. If this did happen, 
    \begin{align*}
        i_{c^-,\gamma(t^+)}\circ i_{\gamma(t^-),c^-} = i_{\gamma(t^-), \gamma(t+)}
    \end{align*}
    would be an isomorphism. However,
    \begin{align*}
        i_{c^+,\gamma(t^+)}\circ i_{\gamma(t^-),c^+} = i_{\gamma(t^-), \gamma(t+)}
    \end{align*}
    can't be an isomorphism, as $i_{\gamma(t^-),c^+}$ is not injective.
    \begin{figure}[ht]
        \centering
        \begin{tikzpicture}[every node/.style={black,right}]

            \draw[ultra thick, name path=lc] (1,6) .. controls (3,3) and (6,1) .. (9,0);
            \draw[ultra thick, name path=lc2] (0,4) .. controls (5,2) .. (10,1.5);
            \draw[ultra thick, green, name path=gamma] (0,0) -- (3,1.5) -- (3.5,3) -- (6,3.5) -- (10,5);
            \draw[ultra thick, green, name path=gamma2] (0,0) -- (3,1.5) -- (5.5,2) -- (6,3.5) -- (10,5);
            
            \node (c-) [circle,fill=cyan,inner sep=1.5pt,label={}] at (3.5, 3) {};
            \node (c-label) [label={}, above=0.5cm of c-] {$c^-$};
            \draw[->, thick, blue] (c-) -- (c-label);

            \node (c+) [circle,fill=cyan,inner sep=1.5pt,label={}] at (5.5, 2) {};
            \node (c+label) [label={}, below=0.5cm of c+] {$c^+$};
            \draw[->, thick, blue] (c+) -- (c+label);

            \node (t-) [circle,fill=cyan,inner sep=1.5pt,label={}] at (3, 1.5) {};
            \node (t-label) [label={}, below=0.5cm of t-] {$\gamma(t^{-}) = \gamma'(t^{-})$};
            \draw[->, thick, blue] (t-) -- (t-label);
            \node (t+) [circle,fill=cyan,inner sep=1.5pt,label={}] at (6, 3.5) {};
            \node (t+label) [label={}, above=0.5cm of t+] {$\gamma(t^{+}) = \gamma'(t^{+})$};
            \draw[->, thick, blue] (t+) -- (t+label);

            \path [name intersections={of=lc and lc2,by=obspt}];
            \node (obs) [circle,fill=red,inner sep=1.5pt,label={}] at (obspt) {};
            
            \path[name path=V] (4.5,0) -- (4.5,5.5);
            \path [name intersections={of=gamma and V,by=gammalabpt}];
            \node (gammalab) [label={}] at (gammalabpt) {};
            \node (gammalabel) [label={}, above=0.5cm of gammalab] {$\gamma$};
            \draw[->, thick, blue] (gammalab) -- (gammalabel);
            
            \path [name intersections={of=gamma2 and V,by=gamma2labpt}];
            \node (gamma2lab) [label={}] at (gamma2labpt) {};
            \node (gamma2label) [label={}, below=0.5cm of gamma2lab] {$\gamma'$};
            \draw[->, thick, blue] (gamma2lab) -- (gamma2label);

            \path[name path=Vmin] (2.5,0) -- (2.5,5.5);
            \path [name intersections={of=lc and Vmin,by=lclabpt}];
            \path[name path=V2min] (2,0) -- (2,5.5);
            \path [name intersections={of=lc2 and V2min,by=lc2labpt}];
            \node (lclab) [label={}] at (lclabpt) {};
            \node (lclabel) [label={}, left=0.5cm of lclab] {$\lcpiece^-_{r}$};
            \draw[->, thick, blue] (lclab) -- (lclabel);
            \node (lc2lab) [label={}] at (lc2labpt) {};
            \node (lc2label) [label={}, left=0.5cm of lc2lab] {$\lcpiece^-_{l}$};
            \draw[->, thick, blue] (lc2lab) -- (lc2label);
            
            \path[name path=Vmax] (7,0) -- (7,5.5);
            \path [name intersections={of=lc and Vmax,by=lcllabpt}];
            \path[name path=V2max] (7.5,0) -- (7.5,5.5);
            \path [name intersections={of=lc2 and V2max,by=lc2llabpt}];
            \node (lcllab) [label={}] at (lcllabpt) {};
            \node (lcllabel) [label={}, below=0.5cm of lcllab] {$\lcpiece^+_{r}$};
            \draw[->, thick, blue] (lcllab) -- (lcllabel);
            \node (lc2llab) [label={}] at (lc2llabpt) {};
            \node (lc2llabel) [label={}, below=0.5cm of lc2llab] {$\lcpiece^+_{l}$};
            \draw[->, thick, blue] (lc2llab) -- (lc2llabel);

            \draw[fill=yellow, opacity=0.2] (obs) circle [radius=2];

        \end{tikzpicture}
        \caption{The green increasing lines indicate the curves $\gamma$ and $\gamma'$
            and the black lines with negative slope are two segments of the Pareto grid.
            The figure indicates the further splitting
            of the two segments into four pieces by the intersection point. 
        }
        \label{fig:id pd fig 2}
    \end{figure}
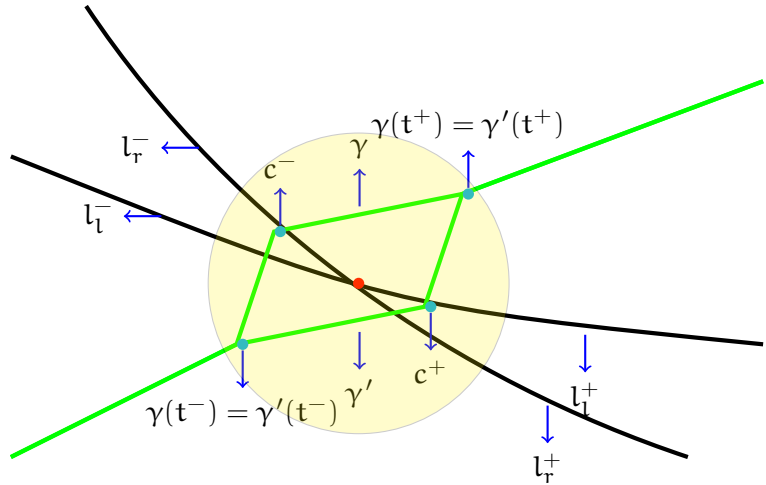

    \textbf{Case 3:} In this final case, we have to consider intersections of
    segments of the same index $k$ as well. Figure \ref{fig:id pd fig 2} serves
    as an illustration for this case as well, with the difference that $l_l$ and $l_r$
    both have index $k$. The usual mapping employed in the earlier cases won't necessarily work 
    here. For instance, if $l^-_l$ cause the birth of a cycle while $l^-_r$ causes
    a death, it could be the case that $l^+_l$ causes a death while $l^+_r$ causes
    a birth. In such situations where the bars flip at a double point of the same
    index, we need to augment the incident segments of the life contour by connecting
    $l^-_l$ with $l^+_r$ and $l^-_r$ with $l^+_l$. After this possible augmentation, the
    earlier map between persistence diagrams work here as well.
\end{proof}

The following proposition clearly follows from the above proof.

\begin{prp}
If the deformation of the increasing curve avoids not just the obstacles, but also the double points of the Pareto grid where the branches of the same index intersect, then the chains representing the elements of each of the persistent homology groups can be transported, consistently, along the deformation.
\end{prp}

As the increasing path crosses a double point of the Pareto grid, the persistent diagram changes continuously, but the birth or death times of a pair of cycles cross, and their corresponding cycles can flip. This phenomenon was noted in \cite{Cer19}. In some sense, this implies that one cannot observe the {\em Elder Rule} \cite{harer2010} functorially along the increasing curve in biparametric persistence.

We can also predict what happens when the curve crosses an obstacle.

\begin{prp}
The increasing curve crossing into (out of) a region with higher number of intersections with the Pareto grid leads to the birth (death) of a $k$-dimensional persistence bar at the the natural height of the corresponding (pseudo)cusp.
\end{prp}
\begin{proof}
    The proof of the result follows from the proof of \ref{p:ind change}. Figure \ref{fig:pf1}
    shows the neighborhood of an obstacle. Crossing an obstacle into the region where the curve intersects
    more with the Pareto grid clearly leads to the creation of a $k$-bar across the
    adjacent segments of the grid.
\end{proof}

\begin{rem}
    After the augmentation possibly required at double points of the same index,
    the collection of augmented segments of the Pareto grid get paired together
    as lines of negative slope meeting at obstacles, across which homology cycles
    are born and killed. A pair of such curves, with index $k$ and $k+1$,
    contains a region $B^k$ enclosed between them.
    The set of these regions $\{B^k_i\}$  are analogous to the barcodes in single parameter persistent
    homology; given points $c_1 \prec c_2 \in V$,
    \begin{align*}
        \rank\left(\left(i_{c_1, c_2}\right)_*:\homol_k(M_{c_1}) \rightarrow \homol_k(M_{c_2})\right) = \#\set[]{ B^k_i \suchthat c_1, c_2 \in B^k_i }
    \end{align*}
\end{rem}

\section{Space of Obstacle Avoiding Curves}
The result of the previous section forces us to concentrate on the spaces of the increasing curves avoiding obstacles: as we established, any two such curves which can be homotopied one into another avoiding obstacles define functorially equivalent persistence diagrams. Therefore, it is important to understand the structure of the components of obstacle avoiding increasing curves.

\subsection{Components of the Obstacle Avoiding Curves}
The problem of enumeration of these components turned out to be remarkably simple.

\begin{prp}\label{prp: avoid} The connected components of a space of increasing curves in the plane avoiding a finite set of obstacles \(o_1,\ldots,o_k\in\Real^2\) are contractible, and their number is equal to the number of {\em chains of obstacles}, i.e. subsets \(o_{i_1}\prec\ldots o_{i_l}\), where \(o_m\prec o_n\) is the vector ordering of the points (empty chain included).
\end{prp}

\begin{proof}
  By turning the $V$ plane by $45^o$, we can identify the increasing paths with the (strictly) Lipschitz functions with constant $L=1$.
  Consider the decomposition of such functions avoiding the obstacles into open components: they are obviously convex.

  For any such component, the pointwise infimum of all trajectories in it (recall that we identify the trajectories with Lipschitz functions $x:[0,T]\to\Real$) is a Lipschitz function passing through some collection of obstacles. Those of the obstacles where $x_\gamma$ is not locally linear necessarily form a chain, which we will refer to as the {\em marker} of the component. This gives a mapping from the components of the space of obstacle avoiding paths to the chains of obstacles.

  To reverse the correspondence, i.e. to associate to a chain a component, pick a small slack $\epsilon$, and for a chain $\gamma$ consider the function
  \[
x_\gamma(t):=\max_{k\in\gamma} \left(x_k+\epsilon-(1-\epsilon)|t-t_k|\right).
  \]
  One can easily see that for small enough $\epsilon$, this function will be in the component whose marker coincides with the chosen chain. This bijection proves the statement.
  \end{proof}

  This correspondence is illustrated in Figure \ref{fig:curves-obs}.

    \begin{figure}[ht]
      \centering
        \includegraphics[width=.8\textwidth]{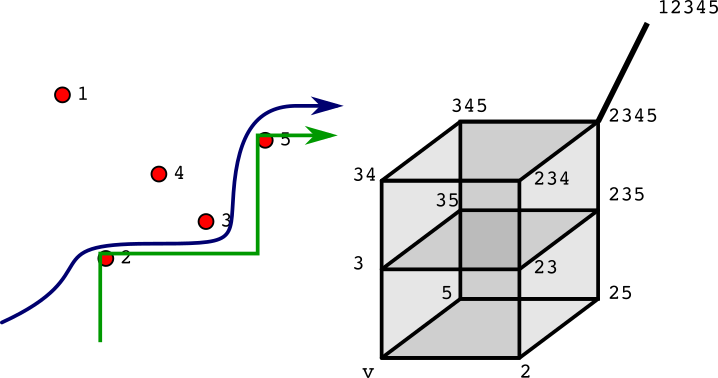}
        \caption{Left display: curves avoiding obstacles and their corresponding chains. The red points are obstacles. The blue curve has the marker shown in green, corresponding to the chain $o_2 \vecl o_5$. \\Right display: cubical complex corresponding to the obstacles on the left display. See Section \ref{sec:cubical} for the explanation of the labels.
        }
        \label{fig:curves-obs}
    \end{figure}

    \subsection{Cubical Complexes}\label{sec:cubical}
    From our structural theorems about the behavior of the persistence diagrams along the increasing paths as the paths cross the obstacles (i.e. cusps or pseudocusps), it is clear that the adjacency of the components of the space of increasing paths is of interest on its own right. Indeed, crossing an obstacle generates or kills a bar on the persistence diagram, and so one can algorithmically generate the persistence diagrams basing on the adjacencies of the components (and the natural heights of the points of the intersection of the increasing curves with the Pareto diagram).

    In the space of increasing curves (which we assume throughout to be parameterized by their natural height, and thus equivalent to the space of $L=1$ Lipschitz functions), the condition of passing through an obstacle forms a codimension $1$ hyperplane.

    The following is quite obvious:
    \begin{prp}
    The multiple intersections of the hyperplanes corresponding to the obstacles are transversal, and contractible.
    \end{prp}

    \begin{proof}
      The transversality follows immediately from the definition; the contractibility from the fact that each component of the set of the increasing curves passing through a given collection of obstacles and avoiding any other obstacle, is convex, if interpreted as functions of the natural height.
    \end{proof}
    
    Consider the cellular complex dual to the space the increasing curves, stratified by the obstacles they pass through. Each of the components becomes a vertex, adjacent cells become connected by an edge etc. The transversality of the intersection implies that the resulting cellular complex is {\em cubical}, that is obtained by the dimension preserving identifications of cubical facets of various cubes.

    \begin{dfn}
We will refer to the complex dual to the stratification of the space of increasing curves by the obstacles they pass through as the {\em obstacle complex}.
      \end{dfn}

    A class of cubical complexes found an extensive use in geometric group theory: namely, the complexes of non-positive curvature (one can turn cubical complexes into path metric spaces by considering flat metric on each of them, in which they are rectangular parallelepipeds (still to be referred to as cubes) whose identified sides have equal lengths). One refers to such metric spaces as CAT($0$) ones, if they satisfy the classical comparison bounds, see \cite{bridson2013metric}. It is well known that a cubical space has non-positive curvature if Gromov's condition is satisfied: in the link of any cube, any three adjacent cubes each sharing a facet will be facets of a common cube, - in other words, nonpositive curvature of a cubical complex depends only on the combinatorial data. A simply-connected cubical complex of nonpositive curvature is referred to as a {\em cubing} (for the details and origins of the nomenclature, see \cite{sageev_ends_1995}).

    \begin{prp} The obstacle complex is a cubing.
      \end{prp}
    \begin{proof}
The only non-immediate fact requiring verification is Gromov's condition, which amounts to the following statement: if for some three obstacles on the plane, there are increasing paths going through each pair of them, then there is a path going through all three. Ordering the obstacles by their natural height makes that obvious. 
      \end{proof}
    
    The cubes of the obstacle complex, as established, correspond to the chains of obstacles, and the dimension of the cubing, - i.e. the highest dimension of the constituent cubes, - equals the length of the longest chain of obstacles. Right display of the Figure \ref{fig:curves-obs} shows the cubing corresponding to the obstacle configuration on the left. We use the convention of \cite{ardila} to mark vertices of the cubing (i.e. increasing obstacle avoiding paths); as one can see (and prove with ease), the markings to markers as defined in Proposition \ref{prp: avoid}.

The configuration of obstacles equips the rectangular parallelepipeds (``cubes'') of a cubing with natural edge lengths (the smallest of the increments of coordinates between two ordered  obstacles in the corresponding chain chain). Continuous deformations of the configurations of obstacles leads to continuous deformation of thus metrized cubing (the distance between which is defined using the Gromov-Hausdorff metric). It seems quite plausible that the function associating such a metric cubings to a smooth map from a manifold to the plane is continuous, providing a version of biparametric stability.

\section{Concluding Discussion}
This note represents just an introduction to the notion of biparametric persistence via increasing paths. We plan to return to the topic addressing some of the results conjectured here (such as stability), and generalizations (to the manifolds with boundary or corners).

      \bibliographystyle{plain}
      \bibliography{references}

\end{document}